\documentclass{amsart}
\usepackage{amssymb, amsmath,amsthm }

\usepackage{hyperref}
\usepackage{stmaryrd}
\usepackage[all]{xy}
\usepackage{dsfont}

\RequirePackage{mathrsfs}
\let\mathcal\mathscr

\theoremstyle{plain}

\newtheorem{prop}{Proposition}[section]

\newtheorem{lem}[prop]{Lemma}
\newtheorem{thm}[prop]{Theorem}

\newtheorem{cor}[prop]{Corollary}

\theoremstyle{remark}

\theoremstyle{definition}
\newtheorem{defi}[prop]{Definition}
\newtheorem{examp}[prop]{Example}

\newtheorem{remark}[prop]{Remark}

\let\cal\mathcal

\def\matrice#1#2#3#4{{\big(\begin{smallmatrix}#1&#2\\ #3&#4\end{smallmatrix}\big)}}

\DeclareMathAlphabet{\mathpzc}{OT1}{pzc}{m}{it}

\DeclareMathOperator{\End}{End}
\DeclareMathOperator{\Hom}{Hom}

\DeclareMathOperator{\GL}{GL}

\DeclareMathOperator{\Ker}{Ker}

\DeclareMathOperator{\Gal}{Gal}

\DeclareMathOperator{\tr}{tr}

\DeclareMathOperator{\Spec}{Spec}
\DeclareMathOperator{\mSpec}{m-Spec}

\DeclareMathOperator{\Ext}{Ext}

\DeclareMathOperator{\p1}{\bold{P}^1}
\DeclareMathOperator{\Mat}{M}

\newcommand{\Q}{\mathbb{Q}}
\newcommand{\Z}{\mathbb Z}
\newcommand{\Qp}{\mathbb {Q}_p}

\newcommand{\qp}{\mathbb{Q}_p}

\newcommand{\qpet}{\mathbb{Q}_p^{\times}}

\newcommand{\Eins}{\boldsymbol{1}}

\newcommand{\ZZ}{\mathbb Z}

\newcommand{\ab}{\mathrm{ab}}

\newcommand{\mm}{\mathfrak m}

\newcommand{\OO}{\mathcal O}

\newcommand{\pp}{\mathfrak p}

\newcommand{\qq}{\mathfrak{q}}

\def\O{{\cal O}}

\let\epsilon\varepsilon

\newcommand{\uX}{\tilde{X}}
\newcommand{\uY}{\tilde{Y}}
\newcommand{\uZ}{\tilde{Z}}
 \newcommand{\univ}{\mathrm{univ}}
 
\newcommand{\nr}{\mathrm{nr}} 

\title[Irreducible components of deformation spaces]{Irreducible components of deformation spaces: wild $2$-adic exercices}
\author{Pierre Colmez}
\address{C.N.R.S., Institut de math\'ematiques de Jussieu, 4 place Jussieu,
75005 Paris, France}
\email{pierre.colmez@imj-prg.fr}
\author{Gabriel Dospinescu}
\address{UMPA, \'Ecole Normale Sup\'erieure de Lyon, 46 all\'ee d'Italie, 69007 Lyon, France}
\email{gabriel.dospinescu@ens-lyon.fr}

\author{Vytautas Pa\v{s}k\={u}nas}
\address{Fakult\"{a}t f\"{u}r Mathematik, Universit\"{a}t Duisburg--Essen, 45117 Essen, Germany}
\email{paskunas@uni-due.de}
\thanks{P.C.~est infinit\'esimalement financ\'e par le projet ArShiFo de l'ANR. V.P.~is partially supported by the DFG, SFB/TR 45. }

\date{\today.}

\begin{document} 
\maketitle

\begin{abstract}
We prove that the irreducible components of the space of framed deformations
of the trivial $2$-dimensional mod~$2$ representation of the absolute Galois group of $\Q_2$ are
in natural bijection with those of the trivial character, confirming a
conjecture of B\"ockle.  We deduce from this result
that crystalline points are Zariski dense in that space: this provides the missing ingredient for
the surjectivity of the $p$-adic local Langlands correspondence for ${\rm GL}_2(\qp)$
in the case $p=2$ (the result was already known for $p>2$).
\end{abstract}

\section{Introduction}
Let $p$ be a prime number and let $L$ be a finite extension of $\qp$.
The $p$-adic local Langlands correspondence from
${\rm  GL}_2(\qp)$ is given by a functor for the category
of unitary admissible $L$-Banach representations of ${\rm  GL}_2(\qp)$
to that of continuous $L$-representations of the absolute Galois
group $G_{\qp}$ of $\qp$. One natural question is whether any $2$-dimensional
$L$-representation of $G_{\qp}$ is in the image (surjectivity of the correspondence).
As was suggested by Kisin, one can try to answer this question
by using the fact that crystalline representations are in the image
and that they form a Zariski dense subset of the space of all representations.
This program was carried out, by two different methods, in~\cite{kisin}
(with some exceptions for $p=2$ and $p=3$) and~\cite{Cbigone} (with the same
exceptions as in~\cite{kisin} for $p=2$), and the upshot is that
surjectivity is known except for $p=2$ and for representations
of $G_{\Q_2}$ whose reduction mod~$2$ is trivial, up to semi-simplification
and torsion by a character.  The aim of this paper is to remove this
exception and hence prove the surjectivity of the $p$-adic local
Langlands correspondence for
${\rm  GL}_2(\qp)$ in full generality. Together with~\cite{PCD} this shows that
the $p$-adic local Langlands correspondence has all the properties that one can wish for.

The methods developed in~\cite{Cbigone} reduce the problem
to that of the Zariski density of crystalline points in the relevant
deformation space (Kisin's approach~\cite{kisin} would require additional information about
$\Ext$ groups of mod~$p$ representations of ${\rm  GL}_2(\qp)$).  
This question can be asked in any dimension and
for representations of the absolute Galois group $G_K$ of a finite extension~$K$
of $\qp$, and the methods developed in~\cite{Ctrianguline,kisin}
(in dimension~$2$) and in~\cite{Chen2,Na2} (for higher dimensions),
make it possible to show that the Zariski closure of the set of
crystalline points is a union of irreducible components.
Deformation spaces are very often irreducible, but Chenevier~\cite{Chen} realized
that for $p=2$ and $K=\Q_2$, this is not the case because the space of deformations
of the determinant is not irreducible.  This led B\"ockle to conjecture
that, in general, irreducible components of the deformation space of a
mod~$p$ representation of $G_K$ are in bijection with those of its determinant, a conjecture
that he verified with Juschka~\cite[cor.~1.8]{BJ} in dimension~$2$, for $p\geq 3$
and also for $p=2$ in the case of the trivial representation and~$K\neq \Q_2$
(see the comments following corollary 1.8 of loc. cit.).  Our main result (theorem~\ref{vienas} below)
is a proof
of this conjecture in the case of the trivial $2$-dimensional representation
of $G_{\Q_2}$.

To state the result more precisely, let us introduce some notations.
Let $\OO$ be the ring of integers of $L$, choose a uniformizer $\varpi$ of $L$
and let $k:=\OO/\varpi$ be the residue field.
  Let $\Eins$ be a one dimensional $k$-vector space on which $G_{\Q_2}$ acts trivially, and let $D_{\Eins}$ be the deformation functor of $\Eins$, and let $D^{\Box}$ 
be the  framed deformation functor of $\Eins\oplus \Eins$, so that for each local artinian $\OO$-algebra $(A, \mm_A)$ with residue field $k$, $D_{\Eins}(A)$ is the set of continuous group 
homomorphisms from $G_{\Q_2}$ to $1+\mm_A$ and $D^{\Box}(A)$ is the set of continuous group 
homomorphisms from $G_{\Q_2}$ to $1+\Mat_2(\mm_A)$. These functors are represented by complete local noetherian $\OO$-algebras $R_{\Eins}$ and $R^{\Box}$ respectively.
Mapping a framed deformation of $\Eins \oplus \Eins$ to its determinant induces a natural transformation $D^{\Box}\rightarrow D_{\Eins}$, and hence a homomorphism of $\OO$-algebras  $d: R_{\Eins}\rightarrow R^{\Box}$. 
 
 \begin{thm}\label{vienas} The map $d: R_{\Eins}\rightarrow R^{\Box}$ induces a bijection between the irreducible components of $\Spec R^{\Box}$ and $\Spec R_{\Eins}$. 
 \end{thm}

Let $\rho^{\univ}: G_{\Q_2}\rightarrow \GL_2(R^{\Box})$ be the universal framed deformation of $\Eins\oplus \Eins$.  If $x$ is a closed point of $\Spec R^{\Box}[1/2]$ then 
its residue field $\kappa(x)$ is a finite extension of~$L$. 
We let $\rho^{\univ}_x: G_{\Q_2}\rightarrow \GL_2(\kappa(x))$ be the representation obtained by specializing the 
universal representation to $x$. We say that a closed point $x$ of $\Spec R^{\Box}[1/2]$ is \textit{crystalline} if $\rho^{\univ}_x$ is a crystalline representation of $G_{\Q_2}$.
Theorem~\ref{vienas} allows us to deduce that every irreducible component of $\Spec R^{\Box}$ contains a crystalline point, such that $\rho^{\univ}_x$ additionally satisfies some 
mild hypothesis (named {\it benign} by Kisin).
 Since we know that the closure of such points in $\Spec R^{\Box}$ is a union of irreducible components, we obtain:

\begin{thm}\label{du} The set of crystalline points in $\Spec R^{\Box}[1/2]$ is dense in $\Spec R^{\Box}$.
\end{thm}

To prove theorem~\ref{vienas} we first produce an explicit presentation, denoted by $S$ in the sequel, of $R^{\Box}$ as a quotient of a formal power series 
ring in $12$ variables over $\OO$ by $4$ relations. The presentation comes from the presentation of the maximal pro-$2$ quotient of $G_{\Q_2}$ as 
a pro-$2$ group with $3$ generators and one relation, see lemma~\ref{present_RBox}. 
We then show that $S$ is Cohen-Macaulay by bounding its dimension by $9$ in \S\ref{complete_intersection} and 
$S[1/2]$ is regular in codimension $1$ by bounding the dimension of the singular locus by $6$ in \S\ref{normalOK}. Serre's criterion allows us to deduce that $S[1/2]$ is normal, 
and so irreducible and connected  components coincide. 
We then show that any irreducible component of $\Spec S$ intersects a well-chosen hypersurface (corollary~\ref{meets})
and we bound (\S\S~\ref{SplusOK} and~\ref{SminusOK}) 
the number of connected components of this hypersurface and then of $\Spec S$ by looking at the chain-connected
components (a $p$-adic avatar of path-connectedness defined in~\S~\ref{ChainsOK}).
This is the  trickiest part of the proof which uses the presentation in an essential way.
 In \S~\ref{density} we deduce theorem~\ref{du} from this.

\begin{remark}
For applications to the $p$-adic Langlands correspondence, treating the case of the trivial residual representation
is enough, but for other questions (for example for global questions) it could be useful to have a result analogous
to theorem~\ref{vienas} for non trivial extensions of the trivial character by itself.
It is quite likely that the methods of this article could be used to prove such a generalization (one would have to
pay more attention to the arcs used in \S\S~\ref{SplusOK} and~\ref{SminusOK}).
\end{remark}  

\section{Notations and preliminaries}

  In this section we introduce notation that will be used in the sequel, and 
  recall some classical commutative algebra results. If $R$ is a commutative ring then we denote by $\Mat_2(R)$ the ring of $2\times 2$-matrices with entries in $R$.
If $A,B\in \GL_2(R)$ then  we let 
$$[A,B]=ABA^{-1}B^{-1}.$$
If $A_1,...,A_k,B_1,...,B_k\in \Mat_2(R)$ we write 
$R/(A_1=B_1,...,A_k=B_k)$ for the quotient of $R$ by the ideal generated by the entries of the matrices
$A_1-B_1,...,A_k-B_k$. In order to avoid confusion we reserve capital letters for the matrices and small letters for the matrix entries.   
Let 
 $$ X=\begin{pmatrix} x_{11} & x_{12} \\ x_{21} & x_{22} \end{pmatrix}, \quad Y=\begin{pmatrix} y_{11} & y_{12} \\ y_{21} & y_{22} \end{pmatrix}, \quad Z=\begin{pmatrix} z_{11} & z_{12} \\ z_{21} & z_{22} \end{pmatrix}. $$
 We consider the matrix entries of $X$, $Y$ and $Z$ as indeterminates and let 
 $$\OO[[X,Y,Z]]:=\OO[[x_{11},x_{12}, x_{21}, x_{22},..., z_{11},..., z_{22}]].$$
 The matrices 
 $$ \uX:= 1+X, \quad \uY:=1+Y, \quad \uZ:=1+Z$$
 are in $\GL_2(\OO[[X,Y,Z]])$. Let 
   $$S:=\OO[[X,Y,Z]]/(\uX^2\uY^4[\uY,\uZ]=1),$$ 
  so that $S$ is a quotient of a formal power series ring over $\OO$ in $12$ variables by $4$ relations. 
  \begin{lem}\label{present_RBox} There is an isomorphism of $\OO$-algebras $R^{\Box}\cong S$. 
  \end{lem}
  \begin{proof} Let $G_{\Q_2}(2)$ be the maximal pro-$2$ quotient of $G_{\Q_2}$. This group is topologically generated by $3$ generators $x, y, z$ and the relation 
  $x^2 y^4 [y, z]=1$, see \cite{Serre_Bourbaki}. Since $\uX$, $\uY$, $\uZ$ are congruent to the identity matrix modulo $\mm_S$, sending $x\mapsto \uX$, 
  $y\mapsto \uY$, $z\mapsto \uZ$ induces a continuous representation $\rho_S: G_{\Q_2}(2)\rightarrow \GL_2(S)$. We consider $\rho_S$ as 
  a framed deformation of $\Eins\oplus \Eins$ to $S$. This gives a homomorphism of $\OO$-algebras $\varphi: R^{\Box}\rightarrow S$.
  
   Let $(A, \mm_A)$ be a local artinian $\OO$-algebra with residue field $k$. The set $D^{\Box}(A)$ is 
  in bijection with the set of continuous group homomorphisms $\rho: G_{\Q_2}\rightarrow 1+\Mat_2(\mm_A)$. 
  Every such $\rho$ factors through the maximal pro-$2$ quotient $G_{\Q_2}(2)$ of $G_{\Q_2}$ as the target has order
  $2^n$ for some natural number $n$. Thus mapping $\rho: G_{\Q_2}(2)\rightarrow 1+\Mat_2(\mm_A)$ to 
  $(\rho(x)-1, \rho(y)-1, \rho(z)-1)$ induces a bijection between the set of such $\rho$ and the set of triples $(X_A, Y_A, Z_A)\in \Mat_2(\mm_A)^3$ 
  satisfying $\uX_A^2\uY_A^4[\uY_A, \uZ_A]=1$, where $\uX_A=1+X_A$, $\uY_A=1+Y_A$ and $\uZ=1+Z_A$.  These are in bijection with 
  $\Hom_{\OO}(S, A)$. Thus $\varphi: R^{\Box}\rightarrow S$ is an isomorphism.
   \end{proof}

   Observe that $\delta=\det\uX\det\uY^2\in S$ satisfies $\delta^2=1$. We let 
   $$S^{+}=S/(\delta+1),\quad S^{-}=S/(\delta-1).$$
    We will write $S^{\pm}$ to mean either of the rings $S^{+}$ or $S^{-}$.

  \begin{lem}\label{idempotent}
   $\Spec S^{\pm}[1/2]$ is an open subset of 
   $\Spec S[1/2]$ and a union of irreducible components of 
   $\Spec S[1/2]$. 
  \end{lem}
  
  \begin{proof}
    Observe that $\frac{1\pm \delta}{2}$ is an idempotent in $S[1/2]$ and 
    $\Spec S^{\pm}[1/2]$ is the zero locus of this idempotent. 
  \end{proof}

        \begin{lem}\label{inversion} Let $(A, \mm_A)$ be a a complete local noetherian ring, and let $a\in \mm_A$. Then 
$\dim A[\frac{1}{a}]\le \dim A -1$. 
If $a$ is not nilpotent and $A$ is equidimensional then $\dim A[\frac{1}{a}]= \dim A -1$.
\end{lem}

\begin{proof} We may assume that $a$ is not nilpotent, so that $A[\frac{1}{a}]\neq 0$. Given a chain of prime ideals in $\Spec A[\frac{1}{a}]$ of length $r$, we may consider them as prime ideals of $\Spec A$, and then extend the chain by including $\mm_A$, thus obtaining a chain of prime ideals in $\Spec A$ of length $r+1$. Hence, $\dim A[\frac{1}{a}]\le \dim A -1$.  

  Suppose further that $A$ is equidimensional. 
  Replacing $A$ by $A/\qq$, where $\qq$ is a minimal prime of $A$ 
   which does not contain $a$, 
  we may assume that $A$ is an integral domain. 
  If $\pp$ is a maximal ideal of $A[\frac{1}{a}]$ then  $\dim A/\pp=1$, (as follows from theorem $146$ in \cite{kapa}).
  Since $A$ is catenary \cite[cor.2.1.13]{bh}, it follows that $\dim A_{\pp}=\dim A-1$.
 Since $\dim A[\frac{1}{a}]\geq \dim A_{\pp}$ (as $a\notin \pp$), the result follows.  
  \end{proof}

  \section{$S$ is complete intersection}\label{complete_intersection}

      \begin{lem}\label{CH}
         Let $A$ be a commutative ring and let $Y\in \Mat_2(A)$. If
         $t=\tr(Y)$ and $d=\det Y$, then
         $$Y^5=(t^4-3dt^2+d^2)Y-dt(t^2-2d)\quad \text{\and}\quad \tr(Y^5)=t(t^4-5t^2d+5d^2).$$
         \end{lem}
\begin{proof} This follows by repeated applications of the Cayley-Hamilton theorem. 
\end{proof}
      
  \begin{prop}\label{CM}
   $S$ is complete intersection of dimension $9$, and 
   $\varpi, \tr(\uX), x_{12}, x_{21}$ form a regular sequence on 
   $S$. In particular, $S$ is flat over $\OO$. 
  \end{prop}
  
\begin{proof} It is enough to prove that the ring $A:=S/(\varpi, \tr(\uX), x_{12}, x_{21})$ is at most $5$-di\-men\-sional, 
 as $S$ is a quotient of a 
$13$-dimensional local regular ring by $4$ equations and we quotient out further by $4$ equations, 
all of them in the maximal ideal. Further, it is enough to prove that $B:= k[[Y, Z]]/ (\uY^4[\uY,\uZ]=\det\uY^2)$ is at most $5$-di\-men\-sio\-nal, as $A$ is finite over $B$.

    Let $\pp$ be a prime ideal of $B$, which we consider as 
    a prime ideal of 
    $k[[Y,Z]]$ containing
    $(\uY^4[\uY,\uZ]=\det\uY^2)$. 
    Since $\uY^5=(\det\uY^2) \uZ\uY\uZ^{-1}$ in 
    $\Mat_2(B/\pp)$, lemma~\ref{CH} yields 
    $(\tr\uY)^3((\tr \uY)^2+\det \uY)\in \pp$.
    Since  $(\tr \uY)^2+\det \uY$ is a unit, we must have 
    $\tr \uY\in \pp$. Using lemma~\ref{CH} again and dividing by the unit 
    $\det\uY^2$, the equation 
  $\uY^5=(\det\uY^2) \uZ\uY\uZ^{-1}$ can be rewritten as
  $\uY\uZ=\uZ\uY$ in $\Mat_2(B/\pp)$. Combined with 
  the equations $2=0$ and $\tr \uY=0$, this yields
  $y_{12}\tr\uZ=0$ and $y_{21}\tr\uZ=0$
  in $B/\pp$. We conclude that the surjection 
  $k[[Y,Z]]\twoheadrightarrow B/\pp$ factors through either 
  $k[[y_{11},Z]]\twoheadrightarrow B/\pp$ or 
  $ \frac{k[[y_{12}, y_{21}, z_{12}, z_{21}]]}{y_{12}z_{21}+y_{21}z_{12}}[[y_{11}, z_{11}]]\twoheadrightarrow B/\pp$. 
  In both cases we obtain $\dim B/\pp\leq 5$ and the result follows. 
 \end{proof}

 \section{$S[1/2]$ is normal}\label{normalOK}

                  Let $x$ be  a closed point of $\Spec S[1/2]$ corresponding to a maximal ideal 
        $m_x$ of $S[1/2]$. Then its residue field $\kappa(x)$ is a finite extension of $L$  
                and the image of 
             $S$ in $\kappa(x)$ is contained in $\OO_{\kappa(x)}$. The resulting morphism
             $S\to \OO_{\kappa(x)}$ is continuous for the $\mm_S$-adic topology on 
             $S$ and the $p$-adic topology on $\OO_{\kappa(x)}$. The universal framed deformation 
              $\rho^{\univ}: G_{\Q_2}\rightarrow \GL_2(S)$ of the trivial representation induces therefore a continuous morphism 
              $\rho^{\univ}_x: G_{\Q_2}\rightarrow \GL_2(\kappa(x))$.

        \begin{lem}\label{ext}
              If $x$ is a singular point of $\Spec S[1/2]$, then there is a character 
              $\delta: G_{\Q_2}\to \OO_{\kappa(x)}^*$ and 
             an exact sequence of $\kappa(x)[G_{\Q_2}]$-modules
              $$0\to \delta\to \rho^{\univ}_x \to \delta \epsilon\to 0,$$
              where $\epsilon: G_{\Q_2}\to \mathbb{Z}_2^{\times}$ is the 
             $2$-adic cyclotomic character. 
          \end{lem}
              
              \begin{proof}
              By lemma 2.3.3 and proposition 
              2.3.5 of \cite{kisinannals}, the $m_x$-adic completion $\hat{S}_x$ of 
              $S[1/2]$ is isomorphic to the framed deformation ring 
              $R^{\square}_x$ of the representation $\rho^{\univ}_x$. 
              Since $x$ is a singular point of $\Spec S[1/2]$, 
              $\hat{S}_x$ is not regular and in particular not formally smooth. 
              Thus the deformation problem for $\rho^{\univ}_x$ is obstructed and hence
             $H^2(G_{\Q_2}, {\rm ad}(\rho^{\univ}_x))\ne 0$.
             By Tate duality this is equivalent to $H^0(G_{\Q_2}, {\rm ad}(\rho^{\univ}_x)\otimes\epsilon)\ne 0$, or
             $\Hom_{G_{\Q_2}}(\rho^{\univ}_x, \rho^{\univ}_x\otimes\epsilon)\ne 0$. In particular, 
             $\rho^{\univ}_x$ is reducible (as $\det \rho^{\univ}_x\ne \det(\rho^{\univ}_x\otimes\epsilon)$) and the result follows then easily. 
                 \end{proof}

    \begin{prop}\label{singular}
      The singular locus of $\Spec S[1/2]$ has dimension $\leq 6$.
    \end{prop}
   \begin{proof} Since  $S[1/2]$ is excellent, the singular locus is closed in $S[1/2]$. Since $S[1/2]$ is Jacobson, this implies that the singular locus is also Jacobson. It follows from lemma~\ref{ext} that all singular closed points of $S[1/2]$ are contained in $V(I)$, where  $I$ is the ideal of $S$ generated by  the elements  
  \begin{equation}\label{trace}
  (\tr \rho^{\univ}(g))^2- (\epsilon(g)+1)^2 \epsilon(g)^{-1} \det \rho^{\univ}(g),
  \end{equation}
as $g$  varies over $G_{\Q_2}$.  Hence, the singular locus is also contained in $V(I)$. Thus it is enough to prove that $\dim S/I\leq 7$, as lemma~\ref{inversion} implies that $\dim (S/I)[1/2]\le 6$. 

  Let $J:= \sqrt{(\varpi, I)}$ and let $\tilde{\rho}: G_{\Q_2} \rightarrow \GL_2(S/J)$ be the representation obtained by reducing the entries of $\rho^{\univ}$ modulo $J$. It is enough
to bound $\dim S/J$ by $6$. Since $\epsilon(g) \equiv 1 \pmod{\varpi}$, we deduce from \eqref{trace} that $(\tr{\rho^{\univ}}(g))^2\equiv 0 \pmod{(\varpi, I)}$, and hence $\tr \tilde{\rho}(g)=0$ for all 
$g\in G_{\Q_2}$. 
Hence the surjection $S\twoheadrightarrow S/J$ factors through: 
\begin{equation}\label{estimate} 
B:=\frac{k[[X, Y, Z]]}{(\det \uX -\det \uY^{-2}, \tr \uX, \tr \uY, \tr \uZ, \tr \uX\uY, \tr \uX\uZ, \tr \uY\uZ)}\twoheadrightarrow S/J
\end{equation}
Let us note that if $\tr \uY=\tr \uZ=0$ then $\tr \uY \uZ= y_{12} z_{21} + y_{21} z_{12}$, as we are in characteristic $2$. Let $I'$ be the ideal   in $k[x_{12}, x_{21}, y_{12}, y_{21}, z_{12}, z_{21}]$ generated by all $2 \times 2$ minors of the matrix
$$\begin{pmatrix} x_{12}& y_{12} & z_{12} \\ x_{21} & y_{21} & z_{21}\end{pmatrix}. $$ 
 It follows from proposition 1.1 in \cite{Bruns_Vettel} that $I'$ defines an irreducible variety of dimension $4$. This implies that
$$A:=\frac {k[[x_{12}, x_{21}, y_{12}, y_{21}, z_{12}, z_{21}]]}{(x_{12} y_{21}+ x_{21}y_{12}, x_{12} z_{21}+ x_{21}z_{12}, y_{12} z_{21}+ y_{21} z_{12})}$$
is $4$-dimensional. The relation $\det \uX -\det \uY^{-2}$ implies  $B$ is finite over $A[[y_{11}, z_{11}]]$ and hence $\dim B=6$. It follows from \eqref{estimate} that $\dim S/J \le 6$.
\end{proof}

      \begin{prop} \label{normal}
      The rings $S[1/2]$ and $S^{\pm}[1/2]$ are normal. 
     \end{prop}
     
     \begin{proof}
     Using lemma~\ref{idempotent}, it suffices to prove that 
     $S[1/2]$ is normal. Since this ring is Cohen-Macaulay by proposition~\ref{CM}, 
it satisfies Serre's condition $S2$. 
     By Serre's criterion of normality, it suffices therefore to prove that 
     $S[1/2]$ is regular in codimension $1$. Since 
     $S[1/2]$ is excellent, the singular locus of 
     $\Spec S[1/2]$ is closed. Since~$\dim S[1/2]=8$ 
(by lemma~\ref{inversion} and taking into account proposition~\ref{CM}), 
proposition~\ref{normal} follows from proposition~\ref{singular}.
          \end{proof}

 \section{A general result about Cohen-Macaulay local rings}
   
    The study of $S^+$ and $S^-$ will be simplified by considering the hyperplane section 
    $\tr(\uX)=0$. The following result is needed in order to control the irreducible
    components of $S$ once we have a good understanding of this hyperplane section. 
   
\begin{prop}\label{contains}
 Let $A$ be a local noetherian ring and suppose that $A$ is Cohen-Macaulay. 
 Let $k\geq 1$ and let $x_1,...,x_k$ be a regular sequence in $A$. Then 
 
 a) Each irreducible component of 
 $\Spec A$ contains an irreducible component of $\Spec A/(x_1,...,x_k)$. 
 
 b) If $k\geq 2$ then each irreducible component of $\Spec A[1/x_k]$ meets the closed subset
  $\Spec (A/(x_1,...,x_{k-1}))[1/x_k]$. 
\end{prop}

\begin{proof}
 a) We start by proving the result for $k=1$. Let $\mathfrak{p}$
  be a minimal prime ideal of $A$. Then $\mathfrak{p}\in {\rm Ass}(A)$, that is there
  exists $f\in A$ such that $\mathfrak{p}={\rm Ann}(f)$. Since 
  $x_1$ is regular, we have ${\rm Ann}(g)={\rm Ann}(gx_1)$ for all 
  $g\in A$. Combining this observation with Krull's intersection theorem,
  we may assume that $x_1$ does not divide $f$. Let 
  $\pi: A\to A/x_1A$ be the canonical projection. Then 
  $\pi(\mathfrak{p})\subset {\rm Ann}(\pi(f))$ and since 
  $\pi(f)\ne 0$, there is an associated prime 
  $\mathfrak{q}$ of $A/x_1A$ such that 
  $\pi(\mathfrak{p})\subset \mathfrak{q}$. 
  Since $x_1$ is regular and $A$ is Cohen-Macaulay, $A/x_1A$ is Cohen-Macaulay and thus $\mathfrak{q}$ is a minimal prime of 
  $A/x_1A$. Since $V(\mathfrak{p})$ contains $\Spec (A/x_1A)/\mathfrak{q}$, 
  we obtain the desired result for $k=1$.
    The general case is proved by induction on $k$, noting that 
    $A/(x_1,...,x_{k-1})$ is Cohen-Macaulay and $x_k$ is regular on it. 
    
  b) Let $I=(x_1,...,x_{k-1})$. By a), it suffices to prove 
  that each irreducible component $C$ of 
  $\Spec A/I$ meets 
  the open subset $D(x_k)$. Since $(x_1, \ldots x_k)$ is a regular sequence in $A$, 
  $\dim A/(x_1, \ldots, x_i)= \dim A-i$, for all $1\le i\le k$.
  Since Cohen-Macaulay rings are equidimensional, we have $\dim C=\dim A/I=\dim A-k+1$.  If $C$ does not meet 
  $D(x_k)$, then 
 $C\subset \Spec A/(x_1,...,x_k)$ and thus $\dim C\leq \dim A-k$, a contradiction. 
 The result follows. 
 \end{proof}

\begin{cor}\label{meets}
Each irreducible component of 
$\Spec S^{\pm}[1/2]$ meets the vanishing locus 
of $\tr(\uX)$, $x_{12}$ and $x_{21}$. 
\end{cor}

\begin{proof}
 Each irreducible component of 
 $\Spec S^{\pm}[1/2]$ is an irreducible component of 
 $\Spec S[1/2]$. The result follows from
 propositions~\ref{contains} and~\ref{CM}.
  \end{proof}

\section{Chains  of closed points}\label{ChainsOK}
  
 Let $A$ be a complete local noetherian $\OO$-algebra, with residue field equal to $k$, and let $X=\Spec A[1/2]$. In the following let
 $K$ be any finite field extension of $L$, $\OO_K$ the ring of integers of $K$ and $\mm_K$ the maximal ideal of $\OO_K$. Let $T_K$ be the Tate algebra in one variable over $K$, that is the ring of power series in $\OO_{K}[[t]][1/2]$, which converge on the whole
 of $\OO_{\mathbb C_2}$, where $\mathbb C_2$ is the $2$-adic completion of the algebraic closure of $\Q_2$.
 
 \begin{defi} We say $x_0, x_1\in X(K)$ are \textit{arc-connected}, if there is an $\OO$-algebra homomorphism $\varphi: A \rightarrow T_K$, such that $x_0: A\rightarrow K$
 is obtained by specializing $\varphi$ at $t=0$ and $x_1: A\rightarrow K$ is obtained  by specializing $\varphi$ at $t=1$. We say that $x_0, x_1\in X(K)$ are \textit{chain-connected}
 if there is a finite sequence of elements of $X(K)$, $x_0=y_0, \ldots, y_n= x_1$, such that $y_{i-1}, y_{i}$ are arc-connected for all $1\le i\le n$.
 \end{defi}

 \begin{examp} If $A=\OO[[x]]$ then every $y\in X(K)$ is arc-connected to $x_0: A\rightarrow K$, $f\mapsto f(0)$.  The arc $\varphi: A\rightarrow T_K$ connecting $x_0$ and $y$ is given by
 $f\mapsto f(\alpha t)$, where $\alpha:=y(x)$. The map is well defined, as $\alpha$ lies in $\mm_K$.
 \end{examp}

 \begin{examp}\label{general example} If $A=\OO[[x_1, \ldots, x_n]]/(f_1,\ldots,  f_m)$ then  $x\mapsto (x(x_1), \ldots, x(x_n))$ induces a bijection between
 $X(K)$ and the set of  $n$-tuples $(\alpha_1, \ldots, \alpha_n)$ of elements in $\mm_K$, such that $f_i(\alpha_1, \ldots, \alpha_n)=0$ for all $0\le i \le m$.
  Let $x, y\in X(K)$ correspond to $n$-tuples  $(\alpha_1, \ldots, \alpha_n)$ and $(\beta_1, \ldots, \beta_n)$ respectively. If we can find topologically nilpotent elements $a_1(t), \ldots, a_n(t)\in T_K$, such that
 $a_j(0)=\alpha_j$, $a_j(1)=\beta_j$ for all $1\le j \le n$, and $f_i(a_1(t), \ldots, a_1(t))=0$ for all $1\le i \le m$ then the arc $\varphi: A\rightarrow T_K$, $f\mapsto f(a_1(t), \ldots, a_n(t))$ connects
 $x$ and $y$.
 \end{examp}
 
 \begin{examp}\label{special example} Let us elaborate on the example~\ref{general example} in our particular situation, when $A=S$.
 The set $X(K)$ is in bijection with triples $(\uX, \uY, \uZ)$ of matrices in $1+\Mat_{2}(\mm_K)$, such that $\uX^2 \uY^4 [\uY, \uZ]=1$. 
 To show that two points corresponding to triplets $(\uX_0, \uY_0, \uZ_0)$ and $(\uX_1, \uY_1, \uZ_1)$ are arc-connected it is enough to produce matrices $(\uX(t), \uY(t), \uZ(t))$
 in $\Mat_{2}(T_K)$  such that the following hold:
 \begin{itemize}
 \item the entries of $\uX-1$, $\uY -1$ and $\uZ -1$ are topologically nilpotent in $T_K$;
 \item   $\uX(t)^2 \uY(t)^4 [\uY(t), \uZ(t)]=1$; 
\item $(\uX_0, \uY_0, \uZ_0)=(\uX(0), \uY(0), \uZ(0))$, $(\uX_1, \uY_1, \uZ_1)=(\uX(1), \uY(1), \uZ(1))$.
\end{itemize}
\end{examp}

 \begin{lem}\label{arc-connect} If $x, y\in X(K)$ are arc-connected then they lie on the same irreducible component of $X$.
 \end{lem}
 \begin{proof} Let $\varphi: A\rightarrow T_K$ be an arc connecting $x$ and $y$. The kernel of $\varphi$ is a prime ideal of $A$ as $T_K$ is an integral  domain. Thus $\Ker \varphi$  contains a minimal  prime $\qq$ of $A[1/2]$, and both $x$ and $y$ lie on $V(\qq)$.
 \end{proof}

 \begin{lem}\label{chain-connect} Assume that distinct  irreducible components of $X$ do not intersect. If $x, y\in X(K)$ are chain-connected then they lie on the same irreducible component
 of $X$.
 \end{lem}
 \begin{proof} The assumption implies that every $x\in X(K)$ lies on a unique irreducible
 component of $X$. The assertion follows from lemma~\ref{arc-connect}.
 \end{proof}

\begin{remark}
We only used disks to define arc-connectedness because this is all that we need,
 but one can also use pieces of analytic curves: i.e.~$\OO$-algebra
morphisms from $A$ to $L\langle x,y\rangle/(f)$, where $f\in L\langle x,y\rangle$ remains irreducible after
any finite extension $L'$ of the field $L$ of coefficients.  This gives more flexibility.
\end{remark}

 \section{$S^+[1/2]$ is an integral domain}\label{SplusOK}

   Let $K$ be a finite extension of $L$ and let
     $$V=\Spec S^+[1/2]/(\tr(\uX)).$$ 
      The main result of this section is

 \begin{prop}\label{prop:chain-connect} Any two points in $V(K)$ are chain-connected in $\Spec S^+[1/2]$.
 \end{prop}

Before proving the proposition, we record a corollary.

\begin{cor}\label{Splus} $S^+[1/2]$ is an integral domain.
\end{cor}
\begin{proof}  It follows from proposition~\ref{normal} that $S^+[1/2]$ is a product of integral domains. Thus it is enough to prove that $X:=\Spec S^+[1/2]$ is irreducible. If $X$ is not irreducible, then we may find a finite extension $K$ of $L$, and points $x, y\in X(K)$ lying on different irreducible components of $X$. It follows
from corollary~\ref{meets} that we may assume that $x$ and $y$ lie in $V(K)$. Proposition~\ref{prop:chain-connect} implies
 that $x$ and $y$ are chain-connected in $X$.
Since $S^+[1/2]$ satisfies the assumption of lemma~\ref{chain-connect}, $x$ and $y$ lie on the same irreducible component
of $X$, giving a contradiction.
\end{proof}

 The proof of proposition~\ref{prop:chain-connect} will be split into several lemmas, in each case producing a triple 
$(\uX(t), \uY(t), \uZ(t))$ as in example~\ref{special example}.  For $\alpha\in\OO$ let 
           $$V_{\alpha}=\Spec S^+[1/2]/(\tr(\uX), \tr(\uY)^2-\alpha\det\uY).$$
                 
     \begin{lem}\label{expliciteq}
       $V(K)$ is the disjoint union of $V_0(K)$ and  $V_4(K)$. Moreover:
       
       a) $V_0(K)$ is the set of triples $(\uX,\uY,\uZ)$ in $(1+M_2(\mm_K))^3$ such that 
       $$\tr(\tilde{Y})=0,\quad \tilde{Y}\tilde{Z}=\tilde{Z}\tilde{Y},\quad \tr(\tilde{X})=0,\quad \det(\tilde{X}\tilde{Y}^2)=-1.$$
       
       b) $V_4(K)$ is the set of triples $(\uX,\uY,\uZ)$ in $(1+M_2(\mm_K))^3$ such that 
        $$\tr(\tilde{Y})^2=4\det \tilde{Y},\quad 
       5\tilde{Y}-2\tr(\tilde{Y})=\tilde{Z}\tilde{Y}\tilde{Z}^{-1},\quad 
      \tr(\tilde{X})=0,\quad \det(\tilde{X}\tilde{Y}^2)=-1.$$
\end{lem}

     \begin{proof} Elements of $V(K)$ are the triples 
     $(\uX,\uY,\uZ)$ in $(1+M_2(\mm_K))^3$ such that 
     $\uX^2\uY^4[\uY,\uZ]=1$, $\det(\uX\uY^2)=-1$ and 
     $\tr(\uX)=0$. Letting $d=\det \uY$, 
     the last two equations and the Cayley-Hamilton theorem yield 
     $\uX^2=\frac{1}{d^2}$, thus the first equation is equivalent to 
     $\tilde{Y}^5=d^2\tilde{Z}\tilde{Y}\tilde{Z}^{-1}$.
     Taking the trace of this equality and using lemma~\ref{CH} we obtain 
     $t(t^2-d)(t^2-4d)=0$, where $t=\tr(\tilde{Y})$. Since 
     $t\in \mm_K$ and
     $d\equiv 1\pmod {\mm_K}$, the last relation is equivalent to 
     $t=0$ or $t^2=4d$. This establishes the first part of the lemma. The 
     description of $V_0(K)$ and $V_4(K)$ follows from what we have already explained and lemma~\ref{CH}. 
     \end{proof}

     \begin{lem}\label{moveX}
       Any two $K$-points of $\Spec \OO[[X]][1/2]/(\tr(\uX), \det\uX+1)$ are chain-connected.
     \end{lem}
     
     \begin{proof}
      It suffices to prove that any $K$-point $\uX$ is chain-connected with the point 
      $\uX_0=\left(\begin{smallmatrix} 1 & 0 \\0 & -1\end{smallmatrix} \right)$. 
      Write 
     $\uX=\left(\begin{smallmatrix} 1+\alpha & \beta \\\gamma & -1-\alpha\end{smallmatrix} \right)$
           with $\alpha,\beta,\gamma\in \mm_K$ and $\alpha(\alpha+2)+\beta\gamma=0$.
           If $v_2(\alpha+2)\leq 1$, use the arc
           $$\tilde{X}(t)=\left(\begin{smallmatrix} 1+t\alpha & t\beta \\\gamma\frac{t\alpha+2}{\alpha+2} & -1-t\alpha\end{smallmatrix} \right)$$
           to connect $\uX$ to some matrix $\left(\begin{smallmatrix} 1 & 0 \\ * & -1\end{smallmatrix} \right)$, then move
           $*$ along a segment to $0$ to reach $\uX_0$.
           
            If
           $v_2(\alpha+2)>1$, then letting $\alpha_1=-\alpha-2$ we still have
           $\alpha_1(\alpha_1+2)+\beta\gamma=0$, and $v_2(\alpha_1+2)=v_2(\alpha)=1$.
           Thus we can apply the previous argument with $\alpha_1$ instead of $\alpha$ and connect
           $\uX$ with $\left(\begin{smallmatrix} -1 & 0 \\ * & 1\end{smallmatrix} \right)$
           for some $*$, then to $\left(\begin{smallmatrix} -1 & 0 \\0 & 1\end{smallmatrix} \right)$. Finally, it remains to see that we can connect
           $\left(\begin{smallmatrix} 1 & 0 \\0 & -1\end{smallmatrix} \right)$ and $\left(\begin{smallmatrix} -1 & 0 \\0 & 1\end{smallmatrix} \right)$.
           We can use the arc $\left(\begin{smallmatrix} 1-2t^2(2-t)^2 & 2t(1-t)(2-t)^2 \\ 2t(1-t)(1+2t-t^2) & -1+2t^2(2-t)^2\end{smallmatrix} \right)$.
           \end{proof}

     \begin{lem}\label{type2}
      Any point of $V_4(K)$ is chain-connected with 
      $(\tilde{X}_0:=\left(\begin{smallmatrix} 1 & 0 \\0 & -1\end{smallmatrix} \right), 1,1)$.
           \end{lem}

     \begin{proof}
       Let $(\tilde{X}, \tilde{Y}, \tilde{Z})$ be a point in $V_4(K)$. 
      The arc
       $$\tilde{Y}(t)=t\tilde{Y}+1-t, \quad \tilde{X}(t)=\frac{\det \tilde{Y}}{\det \tilde{Y}(t)}\uX, \quad
       \tilde{Z}(t)=\tilde{Z}$$
       connects $(\tilde{X}, \tilde{Y}, \tilde{Z})$ with a point of the form 
        $(\tilde{X}_1, 1, \tilde{Z})$ (the reader can check that 
        $(\uX(t), \uY(t), \uZ(t))$ satisfy all equations given in part b) of lemma~\ref{expliciteq}).
 Thus we may assume that $\uY=1$. The arc
        $(\uX, 1, t\uZ+1-t)$ connects $(\uX,1,\uZ)$ to 
        $(\uX,1,1)$. This last point is connected to 
        $(\uX_0,1,1)$ by lemma~\ref{moveX}. 
            \end{proof}

    \begin{lem}
     Any point of $V_0(K)$ is chain-connected with
     $(\tilde{X}_0, \tilde{X}_0, 1)$.
      \end{lem}

    \begin{proof}
       Consider a point $(\tilde{X}, \tilde{Y}, \tilde{Z})$ of $V_0(K)$ and write
       $$\tilde{Y}=\left(\begin{smallmatrix} 1+a & b \\c & -1-a\end{smallmatrix} \right),\quad \tilde{Z}=\left(\begin{smallmatrix} 1+\alpha & \beta \\ \gamma & 1+\delta\end{smallmatrix} \right)$$ with 
       $a,b,c,\alpha,\beta,\gamma, \delta\in \mm_K$.  The condition that
       $\tilde{Y}$ and $\tilde{Z}$ commute is expressed as
       \begin{equation}\label{two}
       b\gamma=c\beta, \quad 2\beta(1+a)=b(\alpha-\delta),\quad 2\gamma(1+a)=c(\alpha-\delta).
       \end{equation}
       Consider now the arc
       $$\tilde{Y}(t)=\left(\begin{smallmatrix} 1+(1-t)a & (1-t)b \\(1-t)c & -1-(1-t)a\end{smallmatrix} \right), \quad
       \tilde{X}(t)=\frac{\det \tilde{Y}}{\det \tilde{Y}(t)}\tilde{X},$$
       and $$\tilde{Z}(t)=\left(\begin{smallmatrix} 1+(1-t)\alpha & (1-t)^2\beta\cdot\frac{1+a}{1+(1-t)a} \\ (1-t)^2\gamma\cdot \frac{1+a}{1+(1-t)a} & 1+(1-t)\delta\end{smallmatrix} \right).$$
        One checks that
        $(\uX(t),\uY(t),\uZ(t))$ satisfy all equations in part a) of lemma~\ref{expliciteq}
 (it is enough to verify that $\uY(t)$ and $\uZ(t)$ commute, which follows by verifying relations \eqref{two}), thus we obtain an arc 
        connecting $(\uX,\uY,\uZ)$ and $(\uX_1, \uX_0,1)$ for some matrix 
        $\uX_1$ with trace $0$ and determinant $-1$. Lemma~\ref{moveX} 
allows us to connect $(\uX_1,\uX_0,1)$ and 
        $(\uX_0,\uX_0,1)$, and the result follows. 
            \end{proof}

    \begin{lem} $(\tilde{X}_0, 1, 1)$ and
     $(\tilde{X}_0, \tilde{X}_0, 1)$ are arc-connected in $\Spec S^+[1/2]$.
    \end{lem}

    \begin{proof}
      Use the arc
      $$
      \tilde{X}(t)=\left(\begin{smallmatrix} 1 & 0 \\0 & -\frac{1}{(2t-1)^2}\end{smallmatrix} \right),\quad \tilde{Y}(t)=\left(\begin{smallmatrix} 1 & 0 \\0 & 2t-1\end{smallmatrix} \right),\quad \uZ(t)=1. $$

    \end{proof}

\section{$S^-[1/2]$ is an integral domain}\label{SminusOK}

     Let $K$ be a finite extension of $L$ containing $\sqrt[4]{-1}$ and $\sqrt{2}$ and let 
     $$V=\Spec \frac{S^-[1/2]}{(\tr(\uX), x_{12}, x_{21})}.$$
     The main result of this section is 
     
     \begin{prop}\label{main}
      Any two points in $V(K)$ belong to the same irreducible component of $\Spec S^-[1/2]$. 
     \end{prop}
     
      As in the previous section, the proposition together with proposition~\ref{contains} and corollary~\ref{meets} imply that 
      \begin{cor}\label{Sminus} $S^{-}[1/2]$ is an integral domain.
      \end{cor}
    
         The proof of proposition~\ref{main} will also be split in a sequence of lemmas. 
         For $\alpha\in \OO$ define 
         $$V_{\alpha}=\Spec \frac{S^-[1/2]}{(\tr(\uX), x_{12}, x_{21}, \tr(\uY)^2-\alpha\det\uY)}.$$
         The proof of the following lemma is identical to that of lemma~\ref{expliciteq} and 
         left to the reader:
         
         \begin{lem} \label{explicit}
          $V(K)$ is the disjoint union of $V_0(K)$ and $V_2(K)$. Moreover, 
           fixing a square root $i$ of $-1$ in $\OO$, we have
          
          a) The points of $V_0(K)$ are the triples 
          $(\uX,\uY,\uZ)$ in $(1+M_2(\mm_K))^3$ with 
          $$\uX=\pm \left(\begin{smallmatrix} \frac{i}{\det \uY} & 0 \\ 0 & -\frac{i}{\det\uY}\end{smallmatrix} \right), \quad 
          \tr(\uY)=0,\quad \uY\uZ+\uZ\uY=0.$$ Any such point automatically satisfies 
          $\tr(\uZ)=0$. 
          
          b) The points of $V_2(K)$ are the triples in $(1+M_2(\mm_K))^3$ with 
          $$\uX=\pm \left(\begin{smallmatrix} \frac{i}{\det \uY} & 0 \\ 0 & -\frac{i}{\det\uY}\end{smallmatrix} \right), \quad 
          \tr(\uY)^2=2\det\uY,\quad \uY\uZ=\uZ\uY.$$
     \end{lem}
         
       \begin{lem}\label{minus}
         Let $(\uX, \uY, \uZ)$ be a point of $V(K)$. Then 
         $(\uX,\uY,\uZ)$ and $(-\uX,\uY,\uZ)$ are arc-connected in 
         $\Spec S^-[1/2]$. 
        \end{lem}
        
        \begin{proof}
          There is $a\in 1+\mm_K$ such that 
          $\uX=\left(\begin{smallmatrix} a & 0 \\ 0 & -a\end{smallmatrix} \right)$.
          At the end of the proof of lemma~\ref{type2} we constructed an arc
          $U(t)$ connecting 
           $\left(\begin{smallmatrix} 1 & 0 \\0 & -1\end{smallmatrix} \right)$ and
            $\left(\begin{smallmatrix} -1 & 0 \\0 & 1\end{smallmatrix} \right)$ and such that 
           $\tr(U(t))=0$ and $\det U(t)=-1$. Simply use the arc 
           $\uX(t)=a U(t)$, $\uY(t)=\uY$ and $\uZ(t)=\uZ$. 
                   \end{proof}
        
        \begin{lem}\label{domain}
        The ring  $$R=\OO[[Y,Z]]/(\tr(\uY), \tr(\uZ), \uY\uZ+\uZ\uY=0)$$
        is an integral domain. 
        \end{lem}
        
        \begin{proof}
        Writing 
      $$\uY= \left(\begin{smallmatrix} 1+a & b \\ c & -1-a\end{smallmatrix} \right), \quad 
      \uZ=\left(\begin{smallmatrix} 1+x & y \\ z & -1-x\end{smallmatrix} \right)$$
      we check that the condition $\uY\uZ+\uZ\uY=0$ is equivalent to 
      $2(1+a)(1+x)+bz+cy=0$. Thus 
      $$R\cong \frac{\OO[[a,b,c,x,y,z]]}{2(1+a)(1+x)+bz+cy}.$$
      It suffices therefore to show that $f=2(1+a)(1+x)+bz+cy$ is irreducible in 
      the unique factorization domain $\OO[[a,b,c,x,y,z]]$. 
      This follows from the fact that the reduction of $f$ modulo $\varpi$ is 
      $bz+cy$, which is irreducible in $k[[b,c,y,z]]$ as can be proved by looking at the
associated graded ring. 
           \end{proof}
        
     \begin{lem}
       Any two points of $V_0(K)$ are on the same irreducible component of 
       $\Spec S^-[1/2]$. 
   \end{lem}
     
     \begin{proof}
       Let $R_{\varepsilon}$ be the quotient of $S^-$ by 
       the ideal generated by the relations $\tr(\uY)=0$, $\tr(\uZ)=0$ and  $\uX= \varepsilon\left(\begin{smallmatrix} \frac{i}{\det\uY} & 0 \\ 0 & -\frac{i}{\det\uY}\end{smallmatrix} \right)$ for $\varepsilon\in\{-1,1\}$. It follows from lemma~\ref{CH}
       that $R_{\varepsilon}$ is isomorphic to the ring $R$ introduced in lemma~\ref{domain}. 
      Thus $R_{\varepsilon}$ is an integral domain and so 
      $\Spec R_{\varepsilon}[1/2]$ is contained in an irreducible component $C_{\varepsilon}$ of 
      $\Spec S^-[1/2]$.  Lemma~\ref{explicit} implies that all points of $V_0(K)$ are $K$-points of 
      $\Spec R_{\varepsilon}[1/2]$ for some $\varepsilon\in\{-1,1\}$, thus all points of 
      $V_0(K)$ belong to $C_1$ or $C_{-1}$. Finally, note that if 
      $(\uX,\uY,\uZ)$ is a $K$-point of $\Spec R_{1}[1/2]$, then 
      $(-\uX,\uY,\uZ)$ is a $K$-point of $\Spec R_{-1}[1/2]$. 
      But by lemma~\ref{minus} these two points belong to the same
      irreducible component of $\Spec S^-[1/2]$. The result follows. 
    \end{proof}
     
     \begin{lem}
      Any two points of $V_2(K)$ are chain-connected in $\Spec S^-[1/2]$. In particular, 
     there is an irreducible component of $\Spec S^-[1/2]$ containing all points of 
     $V_2(K)$. 
      \end{lem}
    
     \begin{proof}
       Any point $(\uX,\uY,\uZ)$ of $V_2(K)$ is arc-connected to the point 
       $(\uX,\uY,\uY)$ via the arc $(\uX(t),\uY(t),\uZ(t))=(\uX,\uY, t\uZ+(1-t)\uY)$, thus we may assume that 
       $\uY=\uZ$.  Let  $\uX_0=\left(\begin{smallmatrix}  1 & 0 \\  0 & -1\end{smallmatrix} \right)$
      and  $\uY_0=\left(\begin{smallmatrix}  1 & 0 \\ 0 & i\end{smallmatrix} \right)$. We will connect 
      $(\uX,\uY,\uY)$ and $(\uX_0,\uY_0,\uY_0)$ using only arcs of the form 
      $(\uX(t), \uY(t),\uY(t))$. 
     
          Write $\uY=\left(\begin{smallmatrix}  a & b \\ c & e\end{smallmatrix} \right)$
       with $b,c\in \mm_K$ and $a,e\in 1+\mm_K$. The condition $\tr(\uY)^2=2\det \uY$ is equivalent to 
       $a^2+e^2+2bc=0$. Using the arc 
       $$\uY(t)=\left(\begin{smallmatrix}  1+(a-1)t & b\frac{1+(a-1)t}{a} \\ c \frac{1+(a-1)t}{a}& e\frac{1+(a-1)t}{a}\end{smallmatrix} \right), \quad 
       \uX(t)=\frac{\det \uY}{\det \uY(t)}\uX,$$
     we may assume that $a=1$. 
     
       Suppose first that $bc\ne 0$. Since 
       $(1+ei)(1-ei)=-2bc$, we have either 
       $v_2(b)<v_2(1+ei)$ or $v_2(c)<v_2(1-ei)$. 
      If $v_2(b)<v_2(1+ei)$ then the arc 
     $$\uY(t)=\left(\begin{smallmatrix}  1 & tb \\  tc+\frac{(t-1)(1+ei)}{b} & te+(1-t)i\end{smallmatrix} \right), \quad 
     \uX(t)=\frac{\det \uY}{\det \uY(t)}\uX$$
connects $(\uX,\uY,\uY)$ to some triple $(\uX_1,  \left(\begin{smallmatrix}  1 & 0 \\  * & i\end{smallmatrix} \right),\left(\begin{smallmatrix}  1 & 0 \\  * & i\end{smallmatrix} \right))$. On the other hand, if $v_2(c)<v_2(1-ei)$, a similar argument shows (replace $e$ by $-e$ and exchange 
     $b$ and $c$) that $(\uX,\uY,\uY)$ is connected to some triple 
     $(\uX_1, \left(\begin{smallmatrix}  1 & *  \\  0 & -i\end{smallmatrix} \right), \left(\begin{smallmatrix}  1 & *  \\  0 & -i\end{smallmatrix} \right))$. 
     
     We may thus assume that $bc=0$ and $\uX$ is diagonal, of trace $0$. 
    Note that $e=\pm i$. If 
      $b=0$, consider the arc 
      $$\uY(t)=\left(\begin{smallmatrix}  1 & 0 \\  tc & e\end{smallmatrix} \right),\quad \uX(t)=\frac{\det \uY}{\det \uY(t)}\uX$$
      to reduce to the case $b=c=0$ and $\uX$ diagonal of trace $0$. A similar argument applies when 
      $c=0$. Thus we are now in the situation $\uY=\uY_0$ or $\uY=\uY_1:=\left(\begin{smallmatrix}  1 & 0 \\ 0 & -i\end{smallmatrix} \right)$, with 
      $\uX$ diagonal of trace $0$. If $\uY=\uY_1$, use the arc 
      $$\uY(t)=\left(\begin{smallmatrix}  1 & 2t(t^2-1)(t^2-2)\\ 2t(t^2-1) & (4t^2-2t^4-1)i\end{smallmatrix} \right), \quad \uX(t)=\frac{\det \uY}{\det \uY(t)}\uX$$
      to reduce to the case $\uY=\uY_0$. Then necessarily $\uX=\uX_0$ or its opposite. We conclude using lemma~\ref{minus}. 
           \end{proof}

     \begin{lem}
      There are points of $V_0(K)$ and $V_2(K)$ respectively 
      which are chain-connected in $\Spec S^-[1/2]$. Consequently, all 
      points of $V(K)$ are on a single irreducible component of $\Spec S^-[1/2]$. 
     \end{lem}
     
     \begin{proof} Let $\rho$ be a square root of $i$, thus $\rho$ is a primitive $8$th root of unity in 
     $K$.
     Let $\alpha$ be a square root of $2$ in $K$ and consider the point 
       $$x=(\uX=\left(\begin{smallmatrix}  -i & -i\alpha  \\ i\alpha & i\end{smallmatrix} \right),\quad \uY=\rho^{-1}\left(\begin{smallmatrix}  -1 & 0 \\ \alpha & 1\end{smallmatrix} \right),\quad 
      \uZ= \left(\begin{smallmatrix}  1 & \alpha \\ 0 & -1\end{smallmatrix} \right)).$$
Note that $\det \uX=-1$, $\det\uY^2=-1$,
$\tr(\uX)=0$, $\uY\uZ+\uZ\uY=0$ and 
$\uY^4=-1$. Thus $x$ is a $K$-point of $\Spec S^-[1/2]$. 
Moreover, since 
$\det(\uX)=-1$ and $\tr(\uX)=0$, lemma~\ref{moveX}
shows that $x$ is chain-connected in $\Spec S^-[1/2]$ with the point 
$$x'=(\left(\begin{smallmatrix}  1 & 0 \\ 0 & -1\end{smallmatrix} \right),\quad \rho^{-1}\left(\begin{smallmatrix}  -1 & 0 \\ \alpha & 1\end{smallmatrix} \right),\quad 
       \left(\begin{smallmatrix}  1 & \alpha \\ 0 & -1\end{smallmatrix} \right)),$$
      which belongs to $V_0(K)$. 
      We will prove that $x$ is also chain-connected with some point of 
      $V_2(K)$, which is enough to conclude. 
      
      Consider the arc 
      $$\uX(t)=\left(\begin{smallmatrix}  i(1-2t^2) & -it\alpha  \\ it\alpha & i\end{smallmatrix} \right),\quad 
      \uY(t)=\rho^{-1}\left(\begin{smallmatrix}  -1 & 0 \\ t\alpha & 1\end{smallmatrix} \right),\quad 
      \uZ(t)=\left(\begin{smallmatrix}  1 & t\alpha  \\ 0 & -1\end{smallmatrix} \right).$$
      We claim that this is an arc in $\Spec S^-[1/2]$. 
       Let 
      $U(t)=\left(\begin{smallmatrix}  -1 & 0 \\ t\alpha & 1\end{smallmatrix} \right)$, then 
      $U(t)^2=\uZ(t)^2=1$, so $\uY(t)^4=-1$ and 
      $$[\uY(t), \uZ(t)]=[U(t),\uZ(t)]=(U(t) \uZ(t))^2.$$
      But by construction $\uX(t)=-i\uZ(t)U(t)$, thus
      $$\uX(t)^2\uY(t)^4[\uY(t),\uZ(t)]=(\uZ(t)U(t))^2 (U(t)\uZ(t))^2=1$$
      and the claim is proved. 
      
      It follows that $x$ is arc-connected with the point 
      $$y=(\uX(0)=\left(\begin{smallmatrix}  i & 0 \\ 0 & i\end{smallmatrix} \right),\quad \uY(0)=\rho^{-1}\left(\begin{smallmatrix}  -1 & 0 \\ 0 & 1\end{smallmatrix} \right),\quad \uZ(0)=\left(\begin{smallmatrix}  1 & 0 \\ 0 & -1\end{smallmatrix} \right)).$$
      
       Finally, consider the arc 
       $$\uX(t)=\uY(t)^{-2},\quad \uY(t)=(1-t)\uY(0)+t \left(\begin{smallmatrix}  1 & 0 \\ 0 & i\end{smallmatrix} \right), \quad 
       \uZ(t)=\left(\begin{smallmatrix}  1 & 0 \\ 0 & -1\end{smallmatrix} \right)$$
       in $\Spec S^-[1/2]$. It connects the point $y$ with the point 
       $$y'=(\left(\begin{smallmatrix}  1 & 0 \\ 0 & -1\end{smallmatrix} \right),\quad \left(\begin{smallmatrix}  1 & 0 \\ 0 & i\end{smallmatrix} \right),\quad
       \left(\begin{smallmatrix}  1 & 0 \\ 0 & -1\end{smallmatrix} \right))$$
       and this point belongs to $V_2(K)$. 
   \end{proof}
  
  \section{Density of crystalline points}\label{density}
 Let $\Eins$ be a one dimensional $k$-vector space on which $G_{\Q_2}$ acts trivially, and let $D_{\Eins}$ be the deformation functor  of $\Eins$. Since $\End_{G_{\Q_2}}(\Eins)=k$ this functor is representable by a complete local noetherian $\OO$-algebra $R_{\Eins}$. We will describe this ring explicitly. Let $\psi^{\univ}: G_{\Q_2}\rightarrow R_{\Eins}^{\times}$ be the universal deformation. Let 
 $\Q_2^{\ab}(2)$ be the compositum in $\overline{\Q}_2$ of all finite abelian extensions $K$ of $\Q_2$, such that $[K:\Q_2]$ is a power of $2$. Then $\Gal(\Q_2^{\ab}(2)/\Q_2)$ is isomorphic 
 to the maximal pro-$2$, abelian quotient of $G_{\Q_2}$, which we denote by $G_{\Q_2}^{\ab}(2)$.
 It follows from local class field theory that 
 $\Q_2^{\ab}(2)$ is the compositum of the $2$-adic cyclotomic extension $\Q_2(\mu_{2^{\infty}})$ and the maximal unramified extension  $\Q_2^{\nr}(2)$ in $\Q_2^{\ab}(2)$. Since 
 $\Q_2(\mu_{2^{\infty}})\cap \Q_2^{\nr}(2)=\Q_2$ and $G_{\Q_2}^{\ab}(2)$ is abelian, we have 
 \begin{equation}\label{affligem1}
 G_{\Q_2}^{\ab}(2)\cong \Gal(\Q_2(\mu_{2^{\infty}})/\Q_2)\times \Gal(\Q_2^{\nr}(2)/\Q_2).
 \end{equation}
Local class field theory  and \eqref{affligem1} give us an isomorphism 
 \begin{equation}\label{affligem2}
 G_{\Q_2}^{\ab}(2)\cong \Z_2^{\times}\times \Z_2\cong 1+ 4 \Z_2 \times \{\pm 1\} \times \Z_2.
 \end{equation}
Thus we may choose $\alpha, \beta, \gamma \in G_{\Q_2}$, such that their images in $ 1+ 4 \Z_2 \times \{\pm 1\} \times \Z_2$ under \eqref{affligem2} are
$(5, 1,0)$, $(1, -1, 0)$ and $(1, 1, 1)$,  respectively. Since $5$ generates $1+4\Z_2$ and $1$ generates $\Z_2$ as topological groups, it follows from \eqref{affligem2} that 
the images of $\alpha$, $\beta$ and $\gamma$ generate $G_{\Q_2}^{\ab}(2)$ as a topological group.
 
 \begin{prop}\label{present_R_1} $R_{\Eins}\cong  \OO[[x, y, z]]/((1+y)^2-1)$.
 \end{prop}
 \begin{proof} Let $(A, \mm_A)$ be a local artinian $\OO$-algebra with residue field $k$. Then $D_{\Eins}(A)$ is in bijection with the set of 
 continuous group homomorphisms $\psi: G_{\Q_2} \rightarrow 1+\mm_A$. Since $1+\mm_A$ is an abelian $2$-group,  any such homomorphism 
 factors through $\psi: G^{\ab}_{\Q_2}(2)\rightarrow 1+\mm_A$. Thus it follows from \eqref{affligem2} that 
 the map $\psi\mapsto (\psi(\alpha)-1, \psi(\beta)-1, \psi(\gamma)-1)$ induces a bijection between the set of such $\psi$ and the set of triples 
 $(a, b, c)\in \mm_A^3$, such that $(1+b)^2=1$, which in turn is in bijection with the set of $\OO$-algebra homomorphisms from 
$\OO[[x, y, z]]/((1+y)^2-1)$ to $A$.
  \end{proof}
 
 \begin{cor}\label{components_R_1} $R_{\Eins}$ is $\OO$-torsion free and has two irreducible components. 
 \end{cor} 
 \begin{proof} The first assertion follows from the fact that $(1+y)^2-1$ is not divisible by $\varpi$ in $\OO[[x, y, z]]$. The two components are given by $y=0$ and $y=-2$.
 \end{proof}
 
 \begin{cor}\label{parity} Let $x, y\in \Spec R_{\Eins}[1/2]$ be closed points, such that $\psi^{\univ}_x$ and $\psi^{\univ}_y$ are crystalline. Then $x$ and $y$ lie on the same irreducible component of 
 $\Spec R_{\Eins}$ if and only if the Hodge--Tate weights of $\psi^{\univ}_x$ and $\psi^{\univ}_y$ have the same parity.
 \end{cor}
 \begin{proof} It follows from the proofs of proposition~\ref{present_R_1} and corollary~\ref{components_R_1} that $x$ and $y$ lie on the same irreducible component if and only if $\psi^{\univ}_x(\beta)=\psi^{\univ}_y(\beta)$. 
 Since both characters are crystalline by assumption, we have $\psi^{\univ}_x= \epsilon^{w_x} \mu$ and $\psi^{\univ}_y= \epsilon^{w_y} \mu'$, where $\epsilon$ is the cyclotomic character, $w_x$, $w_y$ are the Hodge--Tate weights 
and $\mu$, $\mu'$ are unramified characters.  Since $\beta$ was chosen so that its  image in $G^{\nr}_{\Q_2}(2)$ is trivial and $\epsilon(\beta)=-1$, we deduce that $\psi^{\univ}_x(\beta)=(-1)^{w_x}$ and $\psi^{\univ}_y(\beta)=(-1)^{w_y}$.
 \end{proof}
 
Mapping a framed deformation to its determinant induces a natural transformation $D^{\Box}\rightarrow D_{\Eins}$, and hence a homomorphism of $\OO$-algebras  $d: R_{\Eins}\rightarrow R^{\Box}$. 
 \begin{thm}\label{bijection_comp} The map $d: R_{\Eins}\rightarrow R^{\Box}$ induces a bijection between the irreducible components of $\Spec R^{\Box}$ and $\Spec R_{\Eins}$. 
 \end{thm}
 \begin{proof} Since $\det \bigl( \begin{smallmatrix} \psi & 0 \\ 0 & 1 \end{smallmatrix}\bigr)=\psi$, the map $d$ induces a surjection of maximal spectra: $$\mSpec R^{\Box}[1/2]\rightarrow \mSpec R_{\Eins}[1/2].$$  
 Since $R_{\Eins}[1/2]$ is reduced and Jacobson, we conclude that $d: R_{\Eins}[1/2]\rightarrow R^{\Box}[1/2]$ is injective. Let $e=y/2 \in R_{\Eins}[1/2]$, where $y$ is as in proposition~\ref{present_R_1}. Then $e^2=e$, and since $d$ is injective $d(e)$ is a non-trivial idempotent in $R^{\Box}[1/2]$. 
 Since $R^{\Box}[1/2]\cong S[1/2]\cong S^+[1/2] \times S^-[1/2]$, and $S^+[1/2]$ and $S^-[1/2]$ are integral domains by corollaries~\ref{Splus}, \ref{Sminus}, we can conclude 
 that $d$ induces a bijection between the irreducible components of $\Spec R^{\Box}[1/2]$  and $\Spec R_{\Eins}[1/2]$.
 Since  both $R_{\Eins}$ and $R^{\Box}$ are $\OO$-torsion free, this implies the assertion. 
 \end{proof}

 \begin{lem}\label{choose_lattice} Let $K$ be a finite extension of $L$ and let $\rho: G_{\Q_2}\rightarrow \GL_2(K)$ be a continuous representation, such that 
 $\bar{\rho}^{\rm ss}$ is trivial. Then there is an $\OO$-algebra homomorphism $x: R^{\Box} \rightarrow K'$ with $K'$ a finite extension of $K$, such that $\rho^{\univ}_x\cong \rho\otimes_K K'$.
\end{lem} 

\begin{proof} Since $\bar{\rho}^{\rm ss}$ is trivial, we may choose a $K$-basis of the underlying vector space of $\rho$, so that  the image of $\rho$ is contained in 
$\bigl( \begin{smallmatrix} 1+\mm_K & \OO_K \\ \mm_K & 1+\mm_K\end{smallmatrix}\bigr)$. By conjugating $\rho$ with $\bigl( \begin{smallmatrix} \varpi_{K'} & 0 \\ 0& 1\end{smallmatrix}\bigr)$, 
where $K'=K(\sqrt{\varpi_K})$, we may assume that the image of $\rho$ is contained in $\bigl( \begin{smallmatrix} 1+\mm_{K'} & \mm_{K'} \\ \mm_{K'} & 1+\mm_{K'}\end{smallmatrix}\bigr)$.
Hence, $\rho$ is a framed deformation  to $\OO_{K'}$ of the trivial representation of $G_{\Q_2}$ on a $2$-dimensional  $k_K$-vector space. This deformation problem is represented by $\OO_K \otimes_{\OO} R^{\Box}$, and hence we obtain $x: R^{\Box}\rightarrow K'$, such that $\rho^{\univ}_x\cong \rho\otimes_K K'$.
\end{proof}

We say that a closed point $x$ of $\Spec R^{\Box}[1/2]$ is \textit{crystalline} if $\rho^{\univ}_x$ is a crystalline representation of $G_{\Q_2}$.

\begin{thm}\label{dense_crys} The set of crystalline points in $\Spec R^{\Box}[1/2]$ is dense in $\Spec R^{\Box}$.
\end{thm}
\begin{proof} Let $\epsilon$ be the cyclotomic character and let $k\ge 2$ be an integer. The group $H^1_f(\epsilon^k)$ classifying the crystalline 
extensions $0\rightarrow \epsilon^k \rightarrow \rho \rightarrow \Eins\rightarrow 0$ is one dimensional, see for example \cite[Ex. 3.9]{BK}. Thus 
there is a crystalline indecomposable representation $\rho\cong \bigl(\begin{smallmatrix} \epsilon^k & \ast \\ 0 & 1\end{smallmatrix} \bigr)$. Such representation 
is easily seen to be benign in the sense of \cite{kisin}. Since $\det \rho=\epsilon^k$, lemma~\ref{choose_lattice},  theorem~\ref{bijection_comp} and corollary~\ref{parity} imply that both irreducible 
components of $\Spec R^{\Box}$ contain a crystalline benign point.  The closure of such points is a union of irreducible components of $\Spec R^{\Box}$ by  \cite[cor.1.3.4]{kisin}, which allows us to conclude. 
\end{proof}

\begin{remark}
The proof shows that crystalline benign points form a Zariski dense subset of $\Spec R^{\Box}[1/2]$. The subset of 
irreducible and crystalline benign points is still Zariski dense, since the reducible locus of $\Spec R^{\Box}[1/2]$ has dimension at most $7$ 
and $\Spec R^{\Box}[1/2]$ is equidimensional of dimension $8$. 
\end{remark}

\begin{remark}\label{chenevier}
If $\rho$ is an indecomposable $2$-dimensional, characteristic~$2$
representation
of $G_{\Q_2}$ whose semi-simplification is not trivial up to twist by a
character,
then, according to~\cite[cor.~4.2]{Chen},
the generic fiber of the space of framed deformations of $\rho$ has two
irreducible
components, determined by the sign of the determinant on $-1$
(the determinant is viewed as a character of $\Q_2^\times$ via local class
field theory).
It is not difficult to deduce from this the
Zariski density of benign irreducible crystalline points, but this is not
written down explicitly in the literature.
As in the proof of theorem~\ref{dense_crys}, it is enough
to produce one benign crystalline point in each component.

 If $\rho$ is a non trivial extension of $2$ distinct characters,
this can be done
as in the proof of theorem~\ref{dense_crys}.

If $\rho$ is irreducible, we may assume, after twisting by a
character, that $\rho$ is obtained by induction of Serre's fundamental
character $\omega_2$
of the absolute Galois group $G_F$ of the unramified quadratic extension
$F$ of $\Q_2$.
The character $\chi$ attached to a Lubin-Tate formal group for $F$ (and
uniformiser $p$)
is a lifting of $\omega_2$, and so is $\chi^i$ for any integer $i$
congruent to $1$
modulo $3$.  As all the inductions of these characters are benign
crystalline,
it suffices to show that the sign of the determinant on $-1$ of the
induction $V_i$
of $\chi^i$ is $(-1)^i$.  For this, pick $a\in\OO_F^\times$ with norm $-1$
in $\Q_2^\times$,
and lift $a$ to an element $g$ of the inertia of $G_F$. By construction
$g$ maps to $-1$ in
the inertia of $G_{\Q_2}^{\rm ab}$ (isomorphic to $\ZZ_2^\times$
by local class field theory) and the value of the determinant on $-1$ is
the same as its value on $g$. But $V_i$, restricted to $G_F$, is the sum
of $\chi^i$ and its conjugate $\bar\chi^i$ (where $x\mapsto\bar x$
is the non trivial automorphism of $F$). So
$\det_{V_i}(g)= {\rm N}_{F/\Q_2}(\chi^i(g))=
{\rm N}_{F/\Q_2}(a^i)=(-1)^i$.
\end{remark}

\section{Surjectivity of the $p$-adic local Langlands correspondence}
In this paragraph we explain how to adapt the arguments of~\cite[chap.~II]{Cbigone}
to prove surjectivity of the $p$-adic local Langlands correspondence in the cases
not covered by~\cite{Cbigone}.

Let $p$ be any prime, and let $L$ be a finite extension of $\qp$ with ring of integers
$\OO$.
Let $G={\rm GL}_2(\qp)$.
The $p$-adic Langlands correspondence is given by a functor
$\Pi\mapsto{\bf V}(\Pi)$ from the category ${\rm Rep}_L(G)$
of unitary, admissible, $L$-Banach representations of~$G$, residually of finite length, to the category
of continuous, finite dimensional $L$-representations of $G_{\qp}:={\rm Gal}(\overline{\qp}/\qp)$
(see~\cite[chap.~IV]{Cbigone} for the definition of this functor).
We want to prove that any $2$-dimensional $L$-representation~$V$ of $G_{\qp}$
is in the image of this functor, and the strategy used in~\cite[chap.~II]{Cbigone}
(see also~\cite{CD})
is to construct explicitly an object of ${\rm Rep}_L(G)$ mapping to $V$.
We will need to recall some parts of this construction;
this will require
the introduction of a certain number of objects from the theory
of $(\varphi,\Gamma)$-modules.

       Let 
$\OO_{\mathcal{E}}$ be the $p$-adic completion of 
    $\OO[[T]][T^{-1}]$. 
    Let $\Phi\Gamma^{\rm et}$ 
  be the category of \'etale $(\varphi,\Gamma)$-modules\footnote{These are 
$\OO_{\mathcal{E}}$-modules, free of finite rank, endowed with semi-linear commuting actions of 
   Frobenius $\varphi$ and $\Gamma={\rm Gal}(\qp(\mu_{p^{\infty}})/\qp)$, 
   with $\varphi$ of slope $0$.} over $\OO_{\mathcal{E}}$. 
The category $\Phi\Gamma^{\rm et}$ is equivalent by Fontaine's theorem~\cite{FoGrot} to the category 
${\rm Rep}_{\OO}(G_{\qp})$ of $\OO$-modules, free of finite rank,
 with a continuous, $\OO$-linear action of 
$G_{\qp}$. 

To any $D\in \Phi\Gamma^{\rm et}$, one can associate:
\begin{itemize}
\item a compact sub-$\O_L[[T]]$-module $D^\natural$ stable by $\Gamma$ and the canonical left inverse $\psi$
of $\varphi$ (\cite[\S~II.5 and \S~II.6]{Cmirab}; it is immediate, from the definition,
that $(\O_{L'}\otimes_{\O_L}D)^\natural=
\O_{L'}\otimes_{\O_L}D^\natural$ if $L'$ is a finite extension of $L$),

\item for each continuous character $\delta:\qpet\to\OO^\times$, 
a $G$-equivariant sheaf $U\mapsto(D\boxtimes_\delta U)$ on $\p1:=\p1(\qp)$ (cf.~\cite[\S~II.1]{Cbigone} which relies
heavily on~\cite[chap.~V]{Cmirab}).
\end{itemize}

The space $D\boxtimes_\delta\Z_p$ of sections on $\Z_p$ is $D$, and the $G$-equivariance
implies that the space $D\boxtimes_\delta\p1$ of global sections is equipped with
an action of $G$, but this construction gives an interesting $G$-module only if
the pair $(D,\delta)$ is $G$-compatible~\cite{CD}: we denote by $(D^\natural\boxtimes_\delta\p1)_{\rm ns}$
the set of $z\in D\boxtimes_\delta\p1$ such that ${\rm Res}_{\Z_p}\big(\matrice{p^n}{0}{0}{1}z\big)\in
D^\natural$, for all $n\in\ZZ$; this $\O_L$-module is always stable by the upper Borel subgroup of $G$
and we say that {\it $(D,\delta)$ is $G$-compatible} if it is also stable by $G$.

Now, if $V$ is a finite dimensional $L$-representation of $G_{\qp}$, we say that
{\it $(V,\delta)$ is $G$-compatible} if, for one (equivalently any) $\OO$-lattice $V_0$
stable by $G_{\qp}$, $(D(V_0),\delta)$ is $G$-compatible.  If this is the case, we set
$$D(V)\boxtimes_\delta\p1=L\otimes_{\OO}(D(V_0)\boxtimes_\delta\p1), \quad
D(V)^\natural\boxtimes_\delta\p1=L\otimes_{\OO}(D(V_0)^\natural\boxtimes_\delta\p1)_{\rm ns}.$$
The modules $D(V)\boxtimes_\delta\p1$
and $D(V)^\natural\boxtimes_\delta\p1$ do not depend on the choice of $V_0$,
and the quotient $\Pi_\delta(V)$ is an object of ${\rm Rep}_L(G)$ such that
$${\bf V}(\Pi_\delta(V))\cong \check V,$$ where $\check V={\rm Hom}(V,L(1))$ is the Cartier dual of $V$
(this is a translation, via Fontaine's equivalence of categories,
 of~\cite[th.~III.49]{CD} (which is a cleaner version of~\cite[prop.~IV.4.10]{Cbigone})).

So, to prove the surjectivity of the $p$-adic local Langlands correspondence
it suffices, for each $2$-dimensional $L$-representation $V$ of $G_{\qp}$,
to find a character $\delta_V$ such that $(V,\delta_V)$ is $G$-compatible
(if $V$ is absolutely irreducible, \cite[th.~3.5]{PCD} shows that there is at most
one such $\delta_V$). Such a character is given by the following statement
(which is \cite[th.~II.3.1]{Cbigone}, but see (ii) of rem.~II.3.2 and note~7, p.~292, of loc. cit.), in which 
$\epsilon: G_{\qp}\to \OO^{\times}$ is the $p$-adic cyclotomic character.

\begin{thm} 
If $V$ is a $2$-dimensional $L$-representation of $G_{\qp}$,
and if $\delta_V=\epsilon^{-1}\det V$, then $(V,\delta_V)$ is $G$-compatible
and $\Pi(V)=\Pi_{\delta_V}(V)$ is an admissible unitary $L$-Banach space with
central character $\epsilon^{-1}\det V$ such that ${\bf V}(\Pi(V))\cong \check V$.
\end{thm}

\begin{proof}
 Say that $V$ is good if $(V,\epsilon^{-1}\det V)$ is $G$-compatible.
From the discussion before the theorem, it is sufficient
to prove that any $2$-dimensional $V$ is good. The proof
of~\cite{Cbigone} relies on two main ingredients:
\begin{itemize}
\item  Irreducible benign crystalline representations are good (\cite[prop.~II.3.8]{Cbigone}, building upon~\cite{bb}).

\item If $S$ is a  quotient of $\O[[X_1,\dots,X_n]]$ for some $n$, 
which is flat over~ $\OO$, 
if $V$ is a free $S$-module of rank $2$ with
a continuous $S$-linear action of~$G_{\qp}$, 
and if $V_s$ is good for $s$ in a Zariski dense
subset of $\mSpec S[1/p] $, then $V_s$ is good for any $s\in \mSpec S[1/p]$ 
(cf.~\cite[\S~II.3, ${\rm n}^{\rm o}$2]{Cbigone} applied to the $(\varphi,\Gamma)$-module $D(V)$
as defined in~\cite{Dee}).
\end{itemize}

To prove that a given $V$ is good, it suffices therefore to check that benign irreducible crystalline points
are dense (one can ignore the irreducibility condition as reducible representations form a Zariski subset
of codimension~$\geq 1$ in the generic fiber) in the space of framed deformations of the reduction modulo $p$ of a
suitably chosen stable lattice in $V$.  At the time~\cite{Cbigone} was written,
this was known in most cases, but not all: it is easy to check for $p\geq 5$, 
less for $p=3$ in a special case, but this was verified by B\"ockle~\cite{bockle},
and the results of Chenevier~\cite{Chen} allow
 to prove it (cf.~remark~\ref{chenevier}) for $p=2$ if the semi-simplification
of the residual representation of $V$ is not trivial (up to twist by a character).
Now, goodness is invariant by torsion by a character (this follows from~\cite[prop.~II.1.11]{Cbigone})
and $V$ is good if and only if there exists a finite extension $L'$ of $L$
such that $L'\otimes_LV$ is good (this follows from the equality
$(\O_{L'}\otimes_{\O_L}D)^\natural=
\O_{L'}\otimes_{\O_L}D^\natural$).  So theorem~\ref{dense_crys} combined with lemma~\ref{choose_lattice}
allows to conclude also in this case.
\end{proof}

\textit{Acknowledgements.} We thank Gebhard B\"ockle and Ga\"etan Chenevier for a number of stimulating discussions.

\end{document}